\documentclass[a4paper]{amsproc}

\usepackage{amssymb,amsmath,latexsym,amsthm}
\usepackage[english]{babel}
\usepackage[pdftex]{graphicx}
\usepackage{bmpsize}

\newtheorem{theorem}{Theorem}[section]
\newtheorem{lemma}{Lemma}[section]

\newtheorem{corollary}{Corollary}[section]

\theoremstyle{definition}
\newtheorem*{definition}{Definition}
\newtheorem*{remark}{Remark}

\numberwithin{equation}{section}

\begin{document}

\title[On distribution modulo 1 of the sum of powers of a Salem number]{On distribution modulo 1 of the sum of powers of a Salem number}


\author{\sc Dragan Stankov}
\address{Dragan Stankov\\
Katedra Matematike RGF-a\\  
University of Belgrade\\   
Belgrade, \DJ u\v sina 7\\
Serbia}
\email{dstankov@rgf.bg.ac.rs}

\subjclass[2010]{Primary 11K06; Secondary 11R06.}

\maketitle


\begin{abstract}
Let $\theta $ be a Salem number and $P(x)$ a polynomial with integer coefficients. It is well-known that the sequence $(\theta^n)$ modulo 1 is
dense but not uniformly distributed. In this article we discuss the sequence $(P(\theta^n))$ modulo 1. Our
first approach is computational and consists in estimating the number of n so that the fractional part of $(P(\theta^n))$ falls into a subinterval of the partition of $[0,1]$. If Salem number is of degree 4 we can obtain explicit density function of the sequence, using an algorithm which is also given. Some examples confirm that these two approaches give the same result.

\end{abstract}

\bigskip
\section{Introduction}
Studying the distribution modulo 1 of the powers of a fixed real number $\theta $ greater
than 1, has been of interest for some time. In his monograph \cite{Sal}, R. Salem considered the case of
certain special real numbers $\theta $. For instance, he showed that $\{\theta^n\}$ tends to $0$ in $\mathbb{R}/\mathbb{Z}$ when $\theta$ is a Pisot number. If $\theta$ is a Salem
number, $\{\theta^n\}$ is dense in $\mathbb{R}/\mathbb{Z}$, i.e. the fractional parts of $\theta^n$ are dense in the
interval [0, 1) but not uniformly distributed. (See
\cite{Ber}, p. 87-89.) Moreover, Salem numbers are the only known
numbers whose powers are dense in $\mathbb{R}/\mathbb{Z}$.
Recall that a Pisot number is a real algebraic integer greater than 1 whose conjugates other than itself
have modulus less than 1. A Salem number is a real algebraic integer
greater than 1 whose conjugates other than itself have modulus less than
or equal to 1 and at least one conjugate has modulus equal to 1. It is well
known that one and only one of these conjugates $\theta^{-1}$ is inside the unit disc while the
others are on the boundary. The degree $2t$ of $\theta$ is necessarily even and at least
equal to 4.

We will use the following standard notation:
\begin{definition}
For any real number x, we denote:
\begin{enumerate}
\item Integer part of x: $[x] = \max \{ n \in Z : n \le x \}$.
\item Fractional part of x: $\{x\} = x-[x]$.
\item Congruence modulo 1: $x\equiv x'\; \textrm{mod}\; 1 \Leftrightarrow x-x'\in \mathbb{Z}$.
\item Distance from x to the nearest integer:
$\|x\| = \min \{ |x - n| : n \in \mathbb{Z}\}$.
\end{enumerate}
\end{definition}
Let $u =(u(n))$ be an infinite sequence of real numbers. For $0\le x \le 1$, we define the repartition function
\begin{definition}
\[f(x) = \lim_{N\to \infty}\frac{1}{N} \textrm{card}\left\{n<N|\{u(n)\right\}<x\}\]
\end{definition}
There is an analogy between the repartition function and the distribution function as well as
between $f'(x)$ and density function in the Probability theory.

From now on, we suppose that $\theta$ is a Salem number, $u(n)=P(\theta^n)$ where $P(x)$ is a polynomial with integer coefficients.
Denote conjugates of $\theta$ by $\theta^{-1}$, $\exp(\pm2i\pi\omega_1),\ldots ,\exp(\pm2i\pi\omega_{t-1})$.
The sum of an algebraic integer and its conjugates is an integer and therefore for all
$n\in \mathbb{N}$,
\[\theta^n+\theta^{-n}+2\sum_{j=1}^{t-1}\cos2\pi n\omega_j\equiv 0\;\;(\textrm{mod}\;1) \]
so that the distribution of $\theta^n$ (mod 1) is essentially that of $-2\sum_{j=1}^{t-1}\cos2\pi n\omega_j$.
\section{The main theorem}
The aim of this paper is to define an algorithm for determination of the repartition function and its first derivative, in the case $\theta$ is a Salem number of degree four. The algorithm is described in the following
\begin{theorem}\label{glavna}
Let $\theta$ be a Salem number of degree 4, $n$, $m$ positive integer numbers, $P_m(x)=\sum_{j=0}^m a_jx^j$ a polynomial of degree $m$ with integer coefficients. Let $Q^{-1}_1(x),Q^{-1}_2(x),\ldots,Q^{-1}_K(x)$ be all branches of the inverse function of $Q_m(x)=-2\sum_{j=0}^m a_jT_j(x)$, restricted to the domain $[-1,1]$ where $T_j(x)$ is Chebishev polynomial of the first kind of degree $j$.
If domain of $Q_k^{-1}(x)$ is $[\alpha_k ,\beta_k ]$, $\alpha_k<\beta_k$, $k=1,2,\ldots, K$ then we define its extension $S_k(x)$ on $\mathbb{R}$:
\[S_k(x) = \left\{ \begin{array}{c}
Q_k^{-1}(x);\;x\in [\alpha_k,\beta_k]\\
Q_k^{-1}(\alpha_k);\;x\in (-\infty,\alpha_k)\\
Q_k^{-1}(\beta_k);\;x\in (\beta_k,\infty)\\
\end{array}
\right.
,
\]
in the case $Q_k^{-1}(x)$ is decreasing for $x\in [\alpha_k,\beta_k]$,
\[
S_k(x) = \left\{ \begin{array}{c}
-Q_k^{-1}(x);\;x\in [\alpha_k,\beta_k]\\
-Q_k^{-1}(\alpha_k);\;x\in (-\infty,\alpha_k)\\
-Q_k^{-1}(\beta_k);\;x\in (\beta_k,\infty)\\
\end{array}
\right.
\]
in the case $Q_k^{-1}(x)$ is increasing for $x\in [\alpha_k,\beta_k]$.

Let us denote $g(x)=\pi^{-1}\sum_{k=1}^K(\arccos(S_k(x)))$ and let $M$ be a natural number such that $M\geq \max(\{|\alpha_1|,|\alpha_2|,\ldots,|\alpha_K|\}\bigcup\{|\beta_1|,|\beta_2|,\ldots,|\beta_K|\})$.
The repartition function of $u(n)=P_m({\theta}^n)$, is $f(x)=\sum_{i=-M}^{M}(g(x+i)-g(i))$.
The first derivative of the repartition function is
\begin{equation}\label{eq:efprime}
f'(x)=\pi^{-1}\sum_{i=-M}^{M}\sum_{k=1}^K|(\arccos(Q^{-1}_k(x+i)))'|.
\end{equation}
\end{theorem}
\begin{proof}
Let conjugates of $\theta$ be $\theta^{-1}$, $\exp(2i\pi\omega)$, $\exp(-2i\pi\omega)$. Since for any natural $n$
\[a_j(\theta^{nj}+\theta^{-nj}+2\cos2\pi nj\omega)\equiv 0\;\;(\textrm{mod}\;1)\;\;j=0,1,\ldots,m \]
\[P_m(\theta^n)=\sum_{j=0}^m a_j{\theta}^{nj} \equiv -\sum_{j=0}^m a_j(\theta^{-nj}+2\cos2\pi nj\omega) \;\;(\textrm{mod}\;1)\]
so that the distribution of $P_m(\theta^n)$ (mod 1) is essentially that of
\[-2\sum_{j=0}^m a_j\cos2\pi nj\omega=-2\sum_{j=0}^m a_jT_j(\cos2\pi n\omega)=Q_m(\cos2\pi n\omega),\]
where $T_j(x)=\sum_{k=0}^jb^{\langle j\rangle}_kx^k$ is Chebishev polynomial of the first kind. Hence we have
\[Q_m(w)=-2\sum_{j=0}^m a_j\sum_{k=0}^jb^{\langle j\rangle}_k w^k=-2(c_mw^m+c_{m-1}w^{m-1}+\cdots+c_0).\]
where we denoted $w=\cos2\pi n\omega$ and
\begin{equation}\label{eq:coef}
\begin{split}
c_m&=a_mb^{\langle m \rangle}_m,\;\;\\
c_{m-1}&=a_{m-1}b^{\langle m-1 \rangle}_{m-1}+a_mb^{\langle m \rangle}_{m-1},\\
&\ldots\\
c_{j}&=a_{j}b^{\langle j \rangle}_{j}+a_{j+1}b^{\langle j+1 \rangle}_{j}+\cdots+a_mb^{\langle m \rangle}_{j},\\
&\ldots\\
c_0&=a_0b^{\langle 0 \rangle}_0+a_1b^{\langle 1 \rangle}_0+a_2b^{\langle 2 \rangle}_0+\ldots+a_mb^{\langle m \rangle}_0\\
\end{split}
\end{equation}
If we denote an integer $M=2\sum_{j=0}^m|c_j|$ then it is obvious that $Q_m(w)\in[-M,M]$ thus $\{Q_m(w)\}<x\Leftrightarrow$ there is an $i\in\{-M,-M+1,\dots,M\}$ such that $i\leq Q_m(w)<i+x$. Now we conclude that there is a $k\in\{1,2,\dots,K\}$ such that $Q^{-1}_{k}(i)\leq w<Q^{-1}_{k}(i+x)$ or $Q^{-1}_{k}(i)\geq w>Q^{-1}_{k}(i+x)$. Previous two inequalities can be replaced with $S_{k}(i)\geq w>S_{k}(i+x)$. It is easy to verify that $(2\pi)^{-1}\arccos(\cos2\pi t)=\|t\|$ and, as a consequence of this, that $(2\pi)^{-1}\arccos(-\cos2\pi t)=(2\pi)^{-1}\arccos\cos(2\pi t+\pi)=\|t+1/2\|$, for that reason we have
\[(2\pi)^{-1}\arccos(Q^{-1}_{k}(i))\leq \|n\omega\|<(2\pi)^{-1}\arccos(Q^{-1}_{k}(i+x))\]
in the case $Q^{-1}_{k}(i)$ is decreasing, or
\[(2\pi)^{-1}\arccos(-Q^{-1}_{k}(i))\leq \|n\omega+1/2\|<(2\pi)^{-1}\arccos(-Q^{-1}_{k}(i+x))\]
in the case $Q^{-1}_{k}(i)$ is increasing.

It is fulfilled that $\|n\omega\|$ and $\|n\omega+1/2\|$ are uniformly distributed on $[0,1/2]$ because $1$, $\omega$ are \textbf{Q}-linearly independent \cite{Ber}, Theorem 5.3.2 so we can use \cite{Ber} Theorem 4.6.3. Consequently, for all $L$, $R$ such that $0\leq L<R\leq 1/2$,
\[\lim_{N\to \infty}\frac{1}{N} \textrm{card}\{n<N|L\leq \|n\omega\|<R\}=2(R-L),\]
\[\lim_{N\to \infty}\frac{1}{N} \textrm{card}\{n<N|L\leq \|n\omega+\frac{1}{2}\|<R\}=2(R-L).\]
Let $K_1\leq K$ be natural number such that $Q^{-1}_{1},Q^{-1}_{2},\ldots,Q^{-1}_{K_1}$ are decreasing and $Q^{-1}_{K_1+1},Q^{-1}_{K_1+2},\ldots,Q^{-1}_{K}$ are increasing.
Now we can determine the repartition function
\[f(x) = \lim_{N\to \infty}\frac{1}{N} \textrm{card}\{n<N|\{Q_m(\cos(2\pi n\omega))\}<x\}=\]
\[\lim_{N\to \infty}\frac{1}{N} \textrm{card}\bigcup_{i=-M}^{M}\Big(\bigcup_{k=1}^{K_1}\left\{n<N|(2\pi)^{-1}\arccos(Q^{-1}_{k}(i))\right.\leq\]
\[\leq \left.\|n\omega\|<(2\pi)^{-1}\arccos(Q^{-1}_{k}(i+x))\right\}\bigcup\]
\[\bigcup_{k=K_1+1}^K\left\{n<N|(2\pi)^{-1}\arccos(-Q^{-1}_{k}(i))\right.\leq  \]
\[\leq\left.\|n\omega+1/2\|<(2\pi)^{-1}\arccos(-Q^{-1}_{k}(i+x))\right\}\Big)=\]
\[\sum_{i=-M}^{M}\sum_{k=1}^K\left(\pi^{-1}\arccos(S_{k}(i+x))-\pi^{-1}\arccos(S_{k}(i))\right)=\]
\[\sum_{i=-M}^{M}(g(x+i)-g(i))\]
because all sets in the double union are disjoint. Now it is obvious that the first derivative of the repartition function is
\begin{equation}
f'(x)=\sum_{i=-M}^{M}g'(x+i).
\end{equation}
\end{proof}
\begin{remark}
Since $\{P(\theta^n)\}=\{P(\theta^n)+l\},\; l\in\mathbb{Z}$ we can take, without loss of generality, that $a_0=0$.
\end{remark}
\begin{remark}
Since $g(x+i)-g(i)=g(x+i)+c-g(i)-c$ it is clear that $f(x)$ will not be changed if we take $g(x)+c$, $c\in\mathbb{R}$ instead of $g(x)$.
\end{remark}
\begin{remark}
It is necessary to introduce branches of the inverse function $Q(x)$ more precisely. Since $Q(x)$ is a polynomial we can introduce a partition of $[-1,1]$ $-1=x_0<x_1<\cdots<x_K=1$ such that $Q'(x_1)=Q'(x_2)=\cdots=Q'(x_{K-1})=0$ and $Q'(x)$ is positive or negative on each sub-interval $(x_{k-1},x_{k})$, $k=1,2,\ldots,K$. Let us introduce $\alpha_k,\beta_{k}$: if $Q'(x)$ is positive on $(x_{k-1},x_{k})$ then $Q(x_{k-1})=\alpha_k$, $Q(x_{k+1})=\beta_k$;
if $Q'(x)$ is negative on $(x_{k-1},x_{k})$ then $Q(x_{k-1})=\beta_k$, $Q(x_{k})=\alpha_k$. Now we define $Q_k^{-1}(x)$ as the inverse function of $Q(x)$ on $[x_{k-1},x_{k}]$. Let us notice that the first derivative of $Q_k^{-1}(x)$ can tends to infinity only in end points of its domain $[\alpha_k,\beta_{k}]$.
\end{remark}
\begin{corollary}
Let the line $x = v$, $v\in[0,1]$ be a vertical asymptote of the graph of the first derivative of the repartition function $y = f(x)$. $\lim_{x\rightarrow v-0} f'(x)=\infty$ if and only if $v=1$ or $v=\{\beta_k\}$. $\lim_{x\rightarrow v+0} f'(x)=\infty$ if and only if $v=0$ or $v=\{\alpha_k\}$, $k=1,2,\ldots, K$.
\end{corollary}
\begin{proof}
We proved in the Theorem \ref{glavna} that \[f'(x)=\sum_{i=-M}^{M}g'(x+i)=\sum_{i=-M}^{M}\sum_{k=1}^K\pi^{-1}\arccos'(S_{k}(i+x))=\]
\[=-\pi^{-1}\sum_{i=-M}^{M}\sum_{k=1}^K(1-S_{k}^2(i+x))^{-1/2}S_{k}'(i+x)=\]
\[=-\pi^{-1}\sum_{i=-M}^{M}\sum_{k=1}^K(1-(Q^{-1}_{k})^2(i+x))^{-1/2}S_{k}'(i+x)\]

$\lim_{x\rightarrow v-0} f'(x)=\infty$ if and only if there are $i_0$, $k_0$ such that
\[\lim_{x\rightarrow v-0}(1-(Q^{-1}_{k_0})^2(i_0+x))^{-1/2}S_{k_0}'(i_0+x)=\infty.\]
There are two cases: either $Q^{-1}_{k_0}(i_0+v)=\pm 1$ or $\lim_{x\rightarrow v-0}S_{k_0}'(i_0+x)=\infty$. If $Q^{-1}_{k_0}(i_0+v)=\pm 1$ then $i_0-v=Q_{k_0}(\pm 1)$. Since $Q_{k_0}(\pm 1)$ is an integer and $v\in[0,1]$ we conclude that either $v=0$ or $v=1$. But $v=0$ is impossible because $v-0$ will be out of the domain of $Q^{-1}_{k_0}(x)$. If $\lim_{x\rightarrow v-0}S_{k_0}'(i_0+x)=\infty$ then, using the last remark,
either $i_0+v=\alpha_{k_0}$ or $i_0+v=\beta_{k_0}$. Again $i_0+v=\alpha_{k_0}$ is impossible because $i_0+v-0$ will be out of the domain of $ S_{k_0}'(x)$. If $i_0+v=\beta_{k_0}$ then we conclude that $ v=\{\beta_{k_0}\}$ with an exception: if $\beta_{k_0}$ is an integer then its fractional part is $0$ but, as we have seen, $v=0$ is impossible. Nevertheless the claim is true because, in that case the line $x=1$ should be a vertical asymptote of the graph.

It is obvious that $\lim_{x\rightarrow v+0} f'(x)=\infty$ if and only if $v=0$ or $v=\{\alpha_k\}$ can be proved completely analogously.
\end{proof}

We present next procedure for sketching the graph of the first derivative of the repartition function, resulting from the previous corollary:

\begin{enumerate}
\item Find set $\overline{A}$ of local minimum points, set $\overline{B}$ of local maximum points and set $\overline{S}$ of (horizontal inflection) stationary points  of $Q(\cos(t))$.

\item Find set $A$ of fractional parts of values at local minimum points, set $B$ of fractional parts of values at local maximum points and set $S$ of fractional parts of values at stationary points of $Q(\cos(t))$.

\item Let $x_0,x_1,\ldots, x_r$, $0=x_0<x_1<\cdots<x_r=1$ be sorted $r+1$ elements of $A\cup B\cup S$.

\item If $x_i\in A\cup S,\;x_{i+1}\in B\cup S$ then $f'(x)$ has vertical asymptotes $x=x_i$, $x=x_{i+1}$ on interval $(x_i,x_{i+1})$ so $f'(x)$ has the shape of $\cup $.

\item If $x_i\in A\cup S,\;x_{i+1}\in A\setminus(B\cup S)$ then $f'(x)$ has vertical asymptote $x=x_i$ on interval $(x_i,x_{i+1})$ so $f'(x)$ has the shape of left half of $\cup $, we will denote it by $\lfloor$.

\item If $x_i\in B\setminus(A\cup S),\;x_{i+1}\in B\cup S$ then $f'(x)$ has vertical asymptote $x=x_{i+1}$ on interval $(x_i,x_{i+1})$ so $f'(x)$ has the shape of right half of $\cup $, we will denote it by $\rfloor$.

\item If $x_i\in B\setminus(A\cup S),\;x_{i+1}\in A\setminus(B\cup S)$ then $f'(x)$ has no vertical asymptote on interval $(x_i,x_{i+1})$ so $f'(x)$ has the shape of $_\smile$.
\end{enumerate}
\begin{remark}
In the first item of the previous procedure we have to solve the equation $-Q'(\cos t)\sin t=0\Leftrightarrow \sin t=0 \vee Q'(\cos t)=0 $.
Solutions of $\sin t=0$ are $t=k\pi$, $k\in\mathbb{Z}$.
If $x$ is a solution of $Q'(x)=0$ and $-1\le x\le 1$ then $t=\arccos(x)+2k\pi$, $k\in\mathbb{Z}$ is a stationary point of $Q(\cos(t))$.
Thus, if a solution of $Q'(x)=0$ is out of $\mathbb{R}$ or greater than 1 in modulus we should ignore it.

In the second item we have to find $\{Q(\cos k\pi)\}=\{Q(\pm 1)\}=0$, $k\in\mathbb{Z}$ so that $0\in A\cup B\cup S$.
Since the fractional part of a real number is in $[0,1)$ we should take that $0$ and $1$ must be both in or both out of set $A$, as well as $B$ and $S$.
We conclude that $1\in A\cup B\cup S$.
\end{remark}

\section{Linear, quadratic and cubic polynomial}

If $P(x)=a_1 x$ then, using the notation of the Theorem \ref{glavna}, we have only one branch of the inverse function of $Q(x)=-2a_1x$ i.e. $Q^{-1}(x)=\frac{-x}{2a_1}$, so that $g(x)=\pi^{-1}\arccos(\frac{-x}{2a_1})$.
The repartition function is $f(x)=\pi^{-1}\sum_{i=-2a_1}^{2a_1-1}(\arccos(-\frac{x+i}{2a_1})-\arccos(-\frac{i}{2a_1}))$.
The first derivative of the repartition function is
\[f'(x)=\frac{1}{2a_1\pi}\left(\frac{1}{\sqrt{1-\frac{(x-2a_1)^2}{4a_1^2}}}+\frac{1}{\sqrt{1-\frac{(x-2a_1+1)^2}{4a_1^2}}}+\cdots+\frac{1}{\sqrt{1-\frac{(x+2a_1-1)^2}{4a_1^2}}}\right).\]
If we take $a_1=1$ we get Dupain's formulae cited in \cite{DMR}.

Hereafter we suppose that $P(x)=a_2 x^2+a_1 x$ and then, using the notation of the Theorem \ref{glavna}, we have two branches of the inverse function of $Q(x)=-4a_2x^2-2a_1x+2a_2$ i.e. $Q_1^{-1}(x)=-\frac{a_1+\sqrt{a_1^2+8a_2^2-4a_2x}}{4a_2}$, $Q_2^{-1}(x)=-\frac{a_1-\sqrt{a_1^2+8a_2^2-4a_2x}}{4a_2}$. If we denote
$\Delta=a_1^2+8a_2^2 - 4a_2x$;
\[G_1(x)= \left\{ \begin{array}{l}
-\arccos(Q_1^{-1}(x)),\;\; \textrm{if}\; \Delta\ge 0 \;\;\textrm{and} -1\le Q_1^{-1}(x)\le 1,\\
0,\;\; \textrm{othervise};
\end{array}
\right.
\]
\[G_2(x)= \left\{ \begin{array}{l}
-\arccos(Q_2^{-1}(x)),\;\; \textrm{if}\; \Delta\ge 0 \;\;\textrm{and} -1\le Q_2^{-1}(x)\le 1\\
0,\;\; \textrm{othervise};
\end{array}
\right.
\]
then we can prove that
$ g(x)=\pi^{-1}(G_1(x)+G_2(x)-\pi)$ in the case that $x<2a_2+2a_1-4a_2$ and $a_2>0$;
$ g(x)=\pi^{-1}(G_1(x)+G_2(x)+\pi)$ in the case that $x>2a_2+2a_1-4a_2$ and $a_2<0$.
Since
\[G_1'(x)= \left\{ \begin{array}{l}
\frac{2|a_2|}{(16a_2^2-(a_1+(a_1^2+8a_2^2-4a_2x)^{1/2})^2)(a_1^2+8a_2^2-4a_2x)^{1/2}},\;^{\textrm{if}\; \Delta\ge 0,\; |Q_1^{-1}(x)|\le 1}\\
0,\;\; \textrm{othervise};
\end{array}
\right.
\]
\[G_2'(x)= \left\{ \begin{array}{l}
\frac{2|a_2|}{(16a_2^2-(a_1-(a_1^2+8a_2^2-4a_2x)^{1/2})^2)(a_1^2+8a_2^2-4a_2x)^{1/2}},\;^{\textrm{if}\; \Delta\ge 0,\; |Q_2^{-1}(x)|\le 1}\\
0,\;\; \textrm{othervise};
\end{array}
\right.
\]
we can determine $g'(x)$ explicitly by $g'(x)=\pi^{-1}(G_1'(x)+G_2'(x))$.
The repartition function is $f(x)=\sum_{i=-M}^{M}(g(x+i)-g(i))$, $M=4|a_2|+2|a_1|+2|a_2|$ .
The first derivative of the repartition function is  $f'(x)=\sum_{i=-M}^{M}g'(x+i)$.

Since $Q'(-\frac{a_1}{4a_2})=0$ $Q(x)$ has extremum $V=\frac{a_1^2}{4a_2}+2a_2$ at $x=-\frac{a_1}{4a_2}$. Using previous Corollary we can conclude that the graph of the first derivative of the repartition function $y = f'(x)$ has an inner vertical asymptote $x=v$, $v=\{V\}\in(0,1)$ if and only if
\begin{equation}\label{eq:cond}
-1<-\frac{a_1}{4a_2}<1,\;\; a_1\ne 0,\;\; V\notin \mathbb{Z}.
\end{equation}
If $a_2>0$ then $Q(x)$ has maximum $V $ so that $\lim_{x\rightarrow v-0} f'(x)=\infty$. In that case $\lim_{x\rightarrow 0+0} f'(x)=\infty$ so that the graph of $y = f'(x)$ has shape $\cup_\smile$. Similarly if conditions \ref{eq:cond} fulfilled and $a_2<0$ than the graph of
$y = f'(x)$ has shape $_\smile \cup$. If any of the conditions \ref{eq:cond} is not fulfilled then the graph of $y = f'(x)$ has shape $\cup $.

Finally we suppose that $P(x)=a_3x^3+a_2 x^2+a_1 x$ so that we have three branches of the inverse function of $Q(x)=-8a_3x^3-4a_2x^2+(6a_3-2a_1)x+2a_2$. Its explicit formulas are clumsy so we will not cite them here. The roots of $Q'(x)=0$ are
\[ \frac{-a_2\pm\sqrt{9a_3^2+a_2^2-3a_1a_3}}{6a_3}.\]
In the Table \ref{shapetable}, using the procedure for sketching the graph of $f'(x)$, we represent different shapes of graphs.


\begin{table}[ht]
\caption{Procedure for sketching the graph of the first derivative of the repartition function}\label{shapetable}
\renewcommand\arraystretch{1.5}
\noindent\[
\begin{array}{|c|c|c|c|c|c|c|c|c|c|c|c|}
\hline
a_3,a_2,a_1&\overline{x}_1&\overline{x}_2&_{Q(\overline{x}_1)}&_{Q(\overline{x}_2)}&A&B&S&f'(x)\\
\hline\hline
1,1,1&-0.61&0.27&-0.11&2.63&.89,0,1&.63,0,1&&{\bigcup_{\smile}\bigcup}\\
\hline
3,5,6&-0.68&0.12&4.22&10.39&.22,0,1&.39,0,1&&\lfloor\bigcup\rfloor\\
\hline
3,3,10&-0.17&-0.17&6.11&6.11&0,1&0,1&{.11}&\bigcup\bigcup\\
\hline
1,-1,-2&-0.5&0.83&-5&4.48&0,1&.48,0,1&&\bigcup\rfloor\\
\hline
1,2,3&-0.67&0&2.82&4&.82,0,1&0,1&&\lfloor\bigcup\\
\hline
1,-2,-2&-0.39&1.06&-6.21&&.79&0,1&& _{\smile} \bigcup\\
\hline
1,2,-2&-1.06&0.39&&6.21&0,1&.21&&\bigcup_{\smile}\\
\hline
1,0,0&-0.5&0.5&-2&2&0,1&0,1&&\bigcup\\
\hline
1,1,4&\notin \mathbb{R}&\notin \mathbb{R}&&&0,1&0,1&&\bigcup\\
\hline
\end{array}
\]
\end{table}

\section{Some polynomials of degree $m>3$}

If $P(x)=x^m$ then $u(n)=(\theta^m)^n$. We need a Salem's lemma which is proved in \cite{Sal},\cite{Smy}:
\begin{lemma}\label{Salem}
If $\theta $ is a Salem number of degree $2t$ then so is $\theta^m$ for all $m\in \mathbb{N}$.
\end{lemma}
Thus, if we denote $\theta_1=\theta^m$ we get $u(n)=\theta_1^n$ which is a well known sequence. Doche, Mend\`es France and Ruch \cite{DMR} proved next
\begin{lemma}\label{Doche}
Let $\theta_1 $ be a Salem number of degree $2t$, then the repartition function
$f(x)$ of the sequence $(\theta_1^n)$ modulo 1 satisfies
\[f'(x)=1+2\sum_{k=1}^{\infty}J_0(4k\pi)^{t-1}\cos 2\pi kx\]
on $(0,1)$ for all $t\ge 2$.
\end{lemma}
Here $J_0(\cdot)$ is the Bessel function of the first kind of index $0$.

There are integer coefficients $a_j$ of $P(x)$ such that $Q(x)=-2^mx^m$. We can use equations \ref{eq:coef} to find such $a_j$. The power $x^n$ can be expressed in terms of the Chebyshev polynomials of degrees up to $n$ (for proof see \cite{MasHan}, chapter 2.3.1):
\begin{equation}\label{eq:mnadk}
x^m=2^{1-m}\sum_{k=0}^{[m/2]} {^\prime}{m\choose k}
T_{m-2k}(x),
\end{equation}
where the dash $(\sum{^\prime})$ denotes that the $k$-th term in the sum is to be halved if
$m$ is even and $k = m/2$. Let $m$ be odd, we conclude that if
\[P(x)=\sum_{k=0}^{(m-1)/2} {m\choose k} x^{m-2k}
,\]
then $Q(x)=-2^mx^m$, its inverse function can be easily found: $Q^{-1}(x)=-\sqrt[m]{x}/2$. Using \ref{eq:efprime} we obtain that
\[f'(x)=\frac{1}{\pi}\sum_{i=-2^m}^{2^m-1}\frac{\sqrt[m]{x+i}}{2m(x+i)\sqrt{1-(\sqrt[m]{x+i})^2/4}}.\]
We will show that $f'(1/2+x)=f'(1/2-x)$, $|x|\le 1/2$. For that reason $x=1/2$ is a line of symmetry of the graph of $f'(x)$. It is convenient to introduce
\[g'(x)=\frac{1}{\pi}\frac{\sqrt[m]{x}}{2mx\sqrt{1-(\sqrt[m]{x})^2/4}},\]
then we have
\[
\begin{array}{lll}
f'(1/2+x)&=\sum_{i=-2^m}^{2^m-1}g'(1/2+x+i)&(g'(x)\textrm{ is even})\\
&=\sum_{i=-2^m}^{2^m-1}g'(-1/2-x-i)&(i=j-1)\\
&=\sum_{j=-2^m+1}^{2^m}g'(-1/2-x-j+1)&(j=-i)\\
&=\sum_{i=-2^m}^{2^m-1}g'(1/2-x+i)&\\
&=f'(1/2-x).&
\end{array}
\]
Let $m$ be even, we conclude from \ref{eq:mnadk} that if
\[P(x)=\sum_{k=0}^{m/2-1} {m\choose k} x^{m-2k}+\frac{1}{2}{m\choose {m/2}}
,\]
then $Q(x)=-2^mx^m$, its inverse function can be easily found: $Q^{-1}(x)=\pm\sqrt[m]{x}/2$. Using the algorithm presented in the Theorem \ref{glavna} we obtain that
\[f'(x)=\frac{1}{\pi}\sum_{i=1}^{2^m}\frac{\sqrt[m]{-x+i}}{m(-x+i)\sqrt{1-(\sqrt[m]{-x+i})^2/4}}.\]

\begin{center}
\begin{figure}[!htp]
	\includegraphics[width=0.6\linewidth, bb= 0 0 400 300]{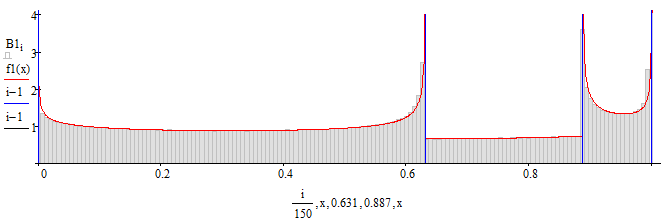}
\caption{$f'(x)$ (the red curve) of $(\{P(\theta^n)\})$, where $P(x)=x^3+x^2+x$ and distribution histogram of the sequence. The shape of $f'(x)$ is $\bigcup_{\smile}\bigcup$.}
\end{figure}
\end{center}

\begin{center}
\begin{figure}[!htp]
    \includegraphics[width=0.6\linewidth, bb= 0 0 400 300]{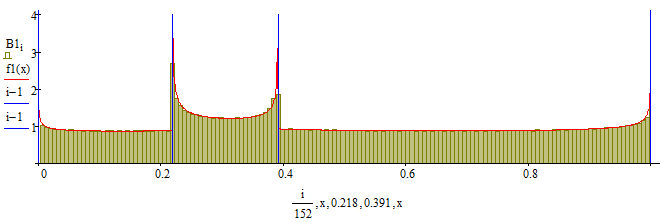}
\caption{$f'(x)$ (the red curve) of $(\{P(\theta^n)\})$, where $P(x)=3x^3+5x^2+6x$ and distribution histogram of the sequence. The shape of $f'(x)$ is $\lfloor\bigcup\rfloor$.}
\end{figure}
\end{center}

\begin{center}
\begin{figure}[!htp]
	\includegraphics[width=0.55\linewidth, bb= 0 0 400 300]{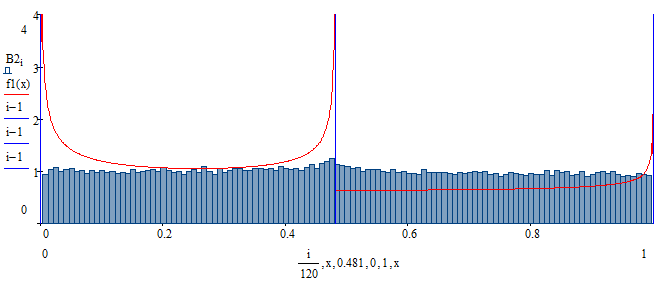}
\caption{$f'(x)$ (the red curve) of $(\{P(\theta^n)\})$, where $P(x)=x^3-x^2-2x$ and $\theta$ is a Salem number of degree 4. The distribution histogram of $(\{P(\theta_1^n)\})$ where $\theta_1$ is a Salem number of degree 6. We can notice that it is close to the density function of the uniform distribution. The shape of $f'(x)$ is $\bigcup\rfloor$.}
\end{figure}
\end{center}

\section{Some examples}

To illustrate the algorithm and the procedure, we give examples of distributions for some polynomials from the Table 1. Using the definition of $f(x)$ and the well known fact that the first derivative of the function $f(x)$ on a small interval $[a,b]$ could be estimated with the finite difference: $f(b)-f(a)$ divided by $b-a$, we will approximate $f'(x)$.
The interval $[0, 1]$ is divided into $p$ pieces.
We compute the fractional part of $P(\theta^n)$ for $1 \leq n \leq N$, and count
the number of $n$ so that the fractional part of $P(\theta^n)$
falls into each of subintervals.
The vertical axis indicates the number of such n divided by $N/p$ so that this normalisation results in a relative histogram that is most similar to $f'(x)$.

Y. Dupain and J. Lesca \cite{DL} conclude that for large degrees t, the sequence $(\{\theta^n\})$
is close to being equidistributed, a fact that S. Akiyama and Y. Tanigawa \cite{AT}
make very explicit in their article. We can notice that the same property is valid for the sequence $(P(\theta^n))$ modulo 1 (Fig. 3).

\end{document}